\documentclass[11pt,a4paper]{amsart}
\usepackage[utf8]{inputenc}
\usepackage{amsmath,amssymb, amsthm, amsfonts, gensymb, commath, float, mathtools, enumerate, breqn}
\usepackage{tabularx}
\usepackage{multirow}
\usepackage{esvect}
\usepackage{subcaption}
\usepackage{tikz}
\usetikzlibrary{positioning, arrows, arrows.meta, cd}
\usetikzlibrary{shapes,backgrounds,shapes,positioning,petri,topaths}

\theoremstyle{theorem}
\newtheorem{theorem}{Theorem}[section]

\newtheorem{proposition}[theorem]{Proposition}

\def\beq#1#2\eeq{%
        \begin{equation}%
        \label{#1}%
            #2%
        \end{equation}%
   }

\usepackage{color}
\usepackage{graphicx}
\usepackage{epstopdf}
\usepackage[text={15cm,23cm}]{geometry} 
\usepackage[colorlinks=true,%
            linkcolor=red!50!black,%
            citecolor=blue!50!black,%
            urlcolor=darkgray]{hyperref}  

\theoremstyle{definition}

\DeclareMathOperator{\tr}{tr}

\title[Markov fractions and Cohn matrices]{Markov fractions and Cohn matrices}

\author{A.P. Veselov}
\address{Department of Mathematical Sciences,
Loughborough University, Loughborough LE11 3TU, UK}
\email{A.P.Veselov@lboro.ac.uk}

\begin{document}

\maketitle

\begin{abstract}
We show that the Markov fractions introduced recently by Springborn coincide with the index of the Cohn matrices defined by Aigner.
This provides a simple concatenation rule for the corresponding continued fractions on the Conway topograph.
We discuss also the $q$-deformation of Markov fractions in the sense of Morier-Genoud and Ovsienko, as well as their ``metallic" version motivated by the work of Spalding and the author.
  \end{abstract}

\section{Introduction}

In 1880 Markov \cite{Markov} introduced his celebrated triples as the solutions of the Diophantine equation
$$x^2 + y^2 + z^2=3xyz.$$ Markov had shown that all its solutions 
in positive integers can be found recursively from $(1,1,1)$ using the Vieta 
involution 
$$
(x,y,z)\mapsto \left(x,y, z'\right),\qquad z' = \frac{x^2+y^2}{z}=3xy-z
$$
and permutation of variables.
He was motivated by the number theory of the binary quadratic forms, but in 1950s Gorshkov \cite{Gorshkov} and Cohn \cite{Cohn} explained the deep connections with the hyperbolic geometry. In the last decades Markov triples became of significance in other parts of mathematics, including algebraic geometry \cite{HP,Rud}.

Recently Springborn \cite{Springborn} introduced the notion of the {\it Markov fractions}, which are also parts of special triples of rationals with the denominators being Markov numbers:
$$
\frac{0}{1}, \frac{1}{2}, \frac{2}{5}, \frac{5}{13}, \frac{12}{29}, \frac{13}{34}, \frac{34}{89}, \frac{70}{169}, \frac{75}{194}, \frac{89}{233}, \frac{179}{433}, \frac{233}{610},  \frac{408}{985},  \dots 
$$
He proved that these fractions have some remarkable Diophantine properties and explained their role in hyperbolic geometry. Their significance for algebraic geometry was demonstrated in \cite{V25}, where it was shown that Markov fractions are precisely all possible slopes of the exceptional vector bundles on $\mathbb P^2.$

In this work we show that Markov fractions coincide with the index of the Cohn matrices defined by Aigner \cite{Aigner}.
On the Conway topograph this leads to a simple concatenation recursion for the corresponding continued fractions, which in a different way were described by Gbur \cite{Gbur}. We show also that a suitable generalisation of Cohn matrices describes in a similar way the $q$-deformation of Markov fractions in the sense of Morier-Genoud and Ovsienko \cite{MGO}.

In the last section we discuss ``metallic" generalisation of the Markov fractions depending on parameter $n \in \mathbb N$, motivated by the work of Spalding and the author \cite{SV1}. The corresponding metallic Markov numbers satisfy the modified Markov equation
$$
X^2 + Y^2 + Z^2=XYZ+4-4n^6
$$
with initial triple $X=Y=n^2+2, \,\, Z=4n^2+2.$
For $n=1$ we have usual Markov triples multiplied by 3.

\section{Markov fractions on the Conway topograph}

The well-known Farey sequence $F_n$ consists of all fractions in $[0,1]$ with denominators not exceeding $n$, e.g.
$$F_5 = \{\frac{0}{1}, \frac{1}{5}, \frac{1}{4}, \frac{1}{3}, \frac{2}{5}, \frac{1}{2}, \frac{3}{5}, \frac{2}{3}, \frac{3}{4}, \frac{4}{5}, \frac{1}{1} \}.$$
Farey noticed that the neighbours  $\frac{a}{b}, \frac{c}{d}$ satisfy the relation $\abs{ad-bc} = 1$ and conjectured that each new term in the Farey sequence is the {\it mediant} $\frac{a}{b}\oplus \frac{c}{d}=\frac{a+b}{c+d}$ of its neighbours.

It is much more convenient, however, to present all fractions in $[0,1]$ using the Conway topograph \cite{Conway}. 

Recall that the Conway topograph is a two-dimensional map with the boundaries forming a trivalent  tree imbedded in the plane. 
Originally Conway labelled the corresponding regions by the so-called lax vectors, but equivalently one can labelled them by the rational numbers using the Farey mediant, see Figure \ref{fig:Rational Conway}.
 Note that the neighbouring regions are occupied by the Farey neighbours $\frac{a}{b}, \frac{c}{d}$ with $\abs{ad-bc} = 1$. 

\begin{figure}[H]
\begin{center}
\begin{tabular}{c c}
	\begin{tikzpicture} 
	
	\node at (1.0,1) {$\frac{p_2}{q_2}$};
	\node at (-1.0,1) {$\frac{p_1}{q_1}$};
	\node at (0,3) {$\frac{p_1+p_2}{q_1+q_2}$};
	

        \draw[thick] (0, 0) -- (0, 2);
        \draw[thick] (0, 1) -- (0, 2);
	\draw[thick] (0,2) -- (1.73,3);
	\draw[thick] (0,2) -- (-1.73,3);

	\end{tikzpicture}

&
\begin{tikzpicture}[scale=0.65, every node/.style={scale=0.75}]
    \node at (1.75,0.75) {\Large$\frac{1}{1}$};
    \node at (-1.75,0.75){\Large$\frac{0}{1}$};
    \node at (0,4) {\Large$\frac{1}{2}$};
    \node at (3.11, 3.8) {\Large$\frac{2}{3}$};
    \node at (2.20, 5.76) {\Large$\frac{3}{5}$}; 
    \node at (4.37, 1.98) {\Large$\frac{3}{4}$}; 
    \node at (0.7, 6.3) {\Large$\frac{4}{7}$}; 
    \node at (3.8, 5.85) {\Large$\frac{5}{8}$}; 
    \node at (5.25, 3.35) {\Large$\frac{5}{7}$}; 
    \node at (4.2, 0.5) {\Large$\frac{4}{5}$}; 
    \node at (-3.11, 3.8) {\Large$\frac{1}{3}$};
    \node at (-2.20, 5.76) {\Large$\frac{2}{5}$}; 
    \node at (-4.37, 2.18) {\Large$\frac{1}{4}$}; 
    \node at (-0.7, 6.3) {\Large$\frac{3}{7}$}; 
    \node at (-3.8, 5.85) {\Large$\frac{3}{8}$}; 
    \node at (-5.25, 3.35) {\Large$\frac{2}{7}$}; 
    \node at (-4.2, 0.5) {\Large$\frac{1}{5}$}; 
    
        
        \draw[thick] (0, -0.5)  -- (0, 2);
    
        \draw[thick] (0,2) -- (1.73,3);
            
            \draw[thick] (1.73, 3) -- (1.99, 4.48);	
    
                \draw[thick] (1.99, 4.48) -- (1.31, 5.53);
                
                    \draw[thick] (1.31, 5.53) -- (0.36, 5.87);
                    \draw[thick] (1.31, 5.53) --  (1.23, 6.53);
                    
                \draw[thick] (1.99, 4.48) -- (2.98, 5.24);	
                
                    \draw[thick] (2.98, 5.24) -- (3.24, 6.21);
                    \draw[thick] (2.98, 5.24) -- (3.98, 5.24);
    
            \draw[thick] (1.73, 3) -- (3.14, 2.49);
    
                \draw[thick] (3.14, 2.49) -- (4.29, 2.97);
    
                    \draw[thick] (4.29, 2.97) -- (4.79, 3.84);
                    \draw[thick] (4.29, 2.97) -- (5.26, 2.71);
                
                \draw[thick] (3.14, 2.49) -- (3.72, 1.38);
                    
                    \draw[thick] (3.72, 1.38) -- (4.63, 0.96);
                    \draw[thick] (3.72, 1.38) -- (3.55, 0.39);
                    
        \draw[thick] (0,2) -- (-1.73,3);
        
            \draw[thick] (-1.73, 3) -- (-1.99, 4.48);	
        
                \draw[thick] (-1.99, 4.48) -- (-1.31, 5.53);
        
                    \draw[thick] (-1.31, 5.53) -- (-0.36, 5.87);
                    \draw[thick] (-1.31, 5.53) -- (-1.23, 6.53);
        
                \draw[thick] (-1.99, 4.48) -- (-2.98, 5.24);	
        
                    \draw[thick] (-2.98, 5.24) -- (-3.24, 6.21);
                    \draw[thick] (-2.98, 5.24) -- (-3.98, 5.24);
        
            \draw[thick] (-1.73, 3) -- (-3.14, 2.49);
        
                \draw[thick] (-3.14, 2.49) -- (-4.29, 2.97);
        
                    \draw[thick] (-4.29, 2.97) -- (-4.79, 3.84);
                    \draw[thick] (-4.29, 2.97) -- (-5.26, 2.71);
        
                \draw[thick] (-3.14, 2.49) -- (-3.72, 1.38);
        
                    \draw[thick] (-3.72, 1.38) -- (-4.63, 0.96);
                    \draw[thick] (-3.72, 1.38) -- (-3.55, 0.39);
\end{tikzpicture}
\end{tabular}
\end{center}
\caption{Conway topograph of rationals in $[0, 1].$} \label{fig:Rational Conway}
\end{figure}

To present in a similar way all the corresponding Markov fractions one should replace here the Farey mediant by the {\it Springborn mediant} \cite{Springborn}
\beq{Sm}
\frac{p_1}{q_1}*\frac{p_2}{q_2}=\frac{p_1q_1+p_2q_2}{q_1^2+q_2^2},\eeq 
or, in the reduced form,
\beq{mfrac}
\frac{p_1}{q_1}*\frac{p_2}{q_2}=\frac{p}{q}, \quad p=\frac{p_1q_1+p_2q_2}{p_2q_1-p_1q_2}, \,\,\, q=\frac {q_1^2+q_2^2}{p_2q_1-p_1q_2}.
\eeq

\begin{figure}[H]
\begin{center}
\begin{tabular}{c c}
	\begin{tikzpicture} 
	
	\node at (1.0,1) {$\frac{p_2}{q_2}$};
	\node at (-1.0,1) {$\frac{p_1}{q_1}$};
	\node at (0,3) {$\frac{p'}{q'}=\frac{p_1q_1+p_2q_2}{q_1^2+q_2^2}$};
	
	\node at (0,5) {$p'=\frac{p_1q_1+p_2q_2}{p_2q_1-p_1q_2}$};
	\node at (0,4) {$q'=\frac{q_1^2+q_2^2}{p_2q_1-p_1q_2}$};

	\draw[thick, ->] (0, 0) -- (0, 1);
        \draw[thick] (0, 1) -- (0, 2);
        \draw[thick] (0, 0) -- (0, 2);
        \draw[thick] (0, 1) -- (0, 2);
	\draw[thick] (0,2) -- (1.73,3);
	\draw[thick] (0,2) -- (-1.73,3);

	\end{tikzpicture}

&
\begin{tikzpicture}[scale=0.65, every node/.style={scale=0.75}]
    \node at (1.75,0.75) {\Large$\frac{1}{2}$};
    \node at (-1.75,0.75){\Large$\frac{0}{1}$};
    \node at (0,4) {\Large$\frac{2}{5}$};
    \node at (3.11, 3.8) {\Large$\frac{12}{29}$};
    \node at (2.20, 5.76) {\Large$\frac{179}{433}$}; 
    \node at (4.37, 1.98) {\Large$\frac{70}{169}$}; 
    \node at (0.7, 6.3) {\Large$\frac{2673}{6466}$}; 
    \node at (3.8, 5.85) {\Large$\frac{15571}{37666}$}; 
    \node at (5.25, 3.35) {\Large$\frac{6089}{14701}$}; 
    \node at (4.2, 0.5) {\Large$\frac{408}{985}$}; 
    \node at (-3.11, 3.8) {\Large$\frac{5}{13}$};
    \node at (-2.20, 5.76) {\Large$\frac{75}{194}$}; 
    \node at (-4.37, 2.18) {\Large$\frac{13}{34}$}; 
    \node at (-0.7, 6.3) {\Large$\frac{1120}{2897}$}; 
    \node at (-3.8, 5.85) {\Large$\frac{2923}{7561}$}; 
    \node at (-5.25, 3.35) {\Large$\frac{507}{1325}$}; 
    \node at (-4.2, 0.5) {\Large$\frac{34}{89}$}; 
    
        
        \draw[thick] (0, -0.5)  -- (0, 2);
    
        \draw[thick] (0,2) -- (1.73,3);
            
            \draw[thick] (1.73, 3) -- (1.99, 4.48);	
    
                \draw[thick] (1.99, 4.48) -- (1.31, 5.53);
                
                    \draw[thick] (1.31, 5.53) -- (0.36, 5.87);
                    \draw[thick] (1.31, 5.53) --  (1.23, 6.53);
                    
                \draw[thick] (1.99, 4.48) -- (2.98, 5.24);	
                
                    \draw[thick] (2.98, 5.24) -- (3.24, 6.21);
                    \draw[thick] (2.98, 5.24) -- (3.98, 5.24);
    
            \draw[thick] (1.73, 3) -- (3.14, 2.49);
    
                \draw[thick] (3.14, 2.49) -- (4.29, 2.97);
    
                    \draw[thick] (4.29, 2.97) -- (4.79, 3.84);
                    \draw[thick] (4.29, 2.97) -- (5.26, 2.71);
                
                \draw[thick] (3.14, 2.49) -- (3.72, 1.38);
                    
                    \draw[thick] (3.72, 1.38) -- (4.63, 0.96);
                    \draw[thick] (3.72, 1.38) -- (3.55, 0.39);
                    
        \draw[thick] (0,2) -- (-1.73,3);
        
            \draw[thick] (-1.73, 3) -- (-1.99, 4.48);	
        
                \draw[thick] (-1.99, 4.48) -- (-1.31, 5.53);
        
                    \draw[thick] (-1.31, 5.53) -- (-0.36, 5.87);
                    \draw[thick] (-1.31, 5.53) -- (-1.23, 6.53);
        
                \draw[thick] (-1.99, 4.48) -- (-2.98, 5.24);	
        
                    \draw[thick] (-2.98, 5.24) -- (-3.24, 6.21);
                    \draw[thick] (-2.98, 5.24) -- (-3.98, 5.24);
        
            \draw[thick] (-1.73, 3) -- (-3.14, 2.49);
        
                \draw[thick] (-3.14, 2.49) -- (-4.29, 2.97);
        
                    \draw[thick] (-4.29, 2.97) -- (-4.79, 3.84);
                    \draw[thick] (-4.29, 2.97) -- (-5.26, 2.71);
        
                \draw[thick] (-3.14, 2.49) -- (-3.72, 1.38);
        
                    \draw[thick] (-3.72, 1.38) -- (-4.63, 0.96);
                    \draw[thick] (-3.72, 1.38) -- (-3.55, 0.39);
\end{tikzpicture}
\end{tabular}
\end{center}
\caption{Conway topograph  $\mathcal T_M$ of Markov fractions.} \label{fig:Markov Conway}
\end{figure}

Juxtaposition of these two topographs establishes the bijection \cite{Springborn}
\beq{bij}
\mu: \mathbb Q\cap [0,1] \to \mathcal {MF}_R,
\eeq
where  $\mathcal {MF}_R$ is the set of Markov fractions between 0 and $1/2,$ which is a fundamental domain of the natural action on the set $\mathcal {MF}$ of all Markov fractions by the integer affine group $\textrm{Aff}_1(\mathbb Z).$

By definition the {\it Springborn function} $\mu(x)$ satisfies the property
\beq{muprop}
\mu\left(\frac{a}{b}\oplus\frac{c}{d}\right)=\mu\left(\frac{a}{b}\right)*\mu\left(\frac{c}{d}\right),\quad |ad-bc|=1,
\eeq
intertwining the Farey and Sprinborn mediants of the neighbours on the Farey tree.

Springborn proved that the Markov fractions $\frac{p}{q}$ together with their properly defined companions are the worst approximable (in certain precise sense) rational numbers (some related results were found earlier in \cite{Gbur}). 

He proved also that the numbers $p,q$, which are defined recursively on the Conway topograph by the local rule (\ref{mfrac}) with the initial data $\frac{0}{1}$ and $\frac{1}{2},$ are coprime integers, with the denominators $q$ being the Markov numbers. This can be derived from the following properties of the Markov fractions (see \cite{Springborn,V25}).  

\begin{figure}[h]
\begin{center}
\includegraphics[width=60mm]{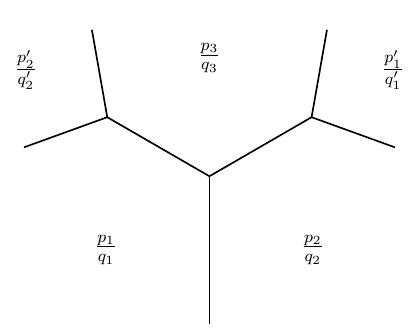} 
\caption{Part of the Conway topograph with Markov fractions.}
\end{center}
\end{figure}

\begin{proposition}
The numbers on a part of the Markov fraction tree shown above satisfy the following relations:
\beq{rel12}
p_2q_3-p_3q_2=q_1,\quad
p_3q_1-p_1q_3=q_2,
\eeq
\beq{rel3}
p_2q_1-p_1q_2=\frac{q_1^2+q_2^2}{q_3}=3q_1q_2-q_3,
\eeq
\beq{rel1'}
p_1'=\frac{p_2q_2+p_3q_3}{q_1}, \quad q_1'=\frac{q_2^2+q_3^2}{q_1},
\eeq
\beq{rel2'}
p_2'=\frac{p_1q_1+p_3q_3}{q_2}, \quad q_2'=\frac{q_1^2+q_3^2}{q_2}.
\eeq
\end{proposition}

In particular, comparing this with the Vieta recursion for the Markov triples, we see that the denominators of the Markov fractions are indeed Markov numbers.

\section{Markov fractions and Cohn matrices}
    
Cohn \cite{Cohn} proposed the following remarkable way to compute Markov numbers.
Instead of the fractions and Farey mediant Cohn used $2\times 2$ matrices with matrix multiplication rule, starting with the special matrices
\begin{equation}
\label{Cohn_1}
A = \begin{pmatrix} 1 & 1 \\ 1 & 2 \end{pmatrix}, \qquad
B = \begin{pmatrix} 3 & 2 \\ 4 & 3 \end{pmatrix}.
\end{equation}
The corresponding Conway topograph $\mathcal T_C$ is shown on the right of Figure \ref{fig:CFTree}.

The relation between Markov and Cohn topographs is given by the trace map \cite{Aigner, SV1}:
$$
C \mapsto m = \frac{1}{3}\tr \, C,
$$
or, by taking the off-diagonal element $m=C_{12}$ (see 
Fig.~\ref{fig:CFTree}).

\begin{figure}[h]
\begin{center}
    \begin{tikzpicture}[scale=0.65, every node/.style={scale=0.75}]
        \node at (1.75,0.75) {$2$};
	\node at (-1.75,0.75){$1$};
	\node at (0,4) {$5$};
	\node at (3.11, 3.8) {$29$};
	\node at (-3.11, 3.8) {$13$};
        \draw[thick] (0, -0.5) -- (0, 2);
		\draw[thick] (0,2) -- (1.73,3);
			\draw[thick] (1.73, 3) -- (1.99, 4.48);	
				\draw[thick] (1.99, 4.48) -- (1.31, 5.53);
				\draw[thick] (1.99, 4.48) -- (2.98, 5.24);	
			\draw[thick] (1.73, 3) -- (3.14, 2.49);
				\draw[thick] (3.14, 2.49) -- (4.29, 2.97);
				\draw[thick] (3.14, 2.49) -- (3.72, 1.38);
		\draw[thick] (0,2) -- (-1.73,3);
			\draw[thick] (-1.73, 3) -- (-1.99, 4.48);	
				\draw[thick] (-1.99, 4.48) -- (-1.31, 5.53);
				\draw[thick] (-1.99, 4.48) -- (-2.98, 5.24);	
			\draw[thick] (-1.73, 3) -- (-3.14, 2.49);
				\draw[thick] (-3.14, 2.49) -- (-4.29, 2.97);
				\draw[thick] (-3.14, 2.49) -- (-3.72, 1.38);
 \begin{scope}[xshift=10cm]
        \node at (-1.75,0.75) {$\displaystyle\left(\begin{array}{cc}
                                                1 & 1 \\
                                                1 & 2 
                                              \end{array}\right)$};
	\node at (1.75,0.75){$\displaystyle\left(\begin{array}{cc}
                                                3 & 2 \\
                                                4 & 3 
                                              \end{array}\right)$};
	\node at (0,4) {$\displaystyle\left(\begin{array}{cc}
                                                7 & 5 \\
                                               11 & 8 
                                              \end{array}\right)$};
	\node at (3.2, 3.8) {$\displaystyle\left(\begin{array}{cc}
                                               41 & 29 \\
                                               65 & 46
                                              \end{array}\right)$};
	\node at (-3.2, 3.8) {$\displaystyle\left(\begin{array}{cc}
                                               18 & 13 \\
                                               29 & 21
                                              \end{array}\right)$};
        \draw[thick] (0, -0.5) -- (0, 2);
		\draw[thick] (0,2) -- (1.73,3);
			\draw[thick] (1.73, 3) -- (1.99, 4.48);	
				\draw[thick] (1.99, 4.48) -- (1.31, 5.53);
				\draw[thick] (1.99, 4.48) -- (2.98, 5.24);	
			\draw[thick] (1.73, 3) -- (3.14, 2.49);
				\draw[thick] (3.14, 2.49) -- (4.29, 2.97);
				\draw[thick] (3.14, 2.49) -- (3.72, 1.38);
		\draw[thick] (0,2) -- (-1.73,3);
			\draw[thick] (-1.73, 3) -- (-1.99, 4.48);	
				\draw[thick] (-1.99, 4.48) -- (-1.31, 5.53);
				\draw[thick] (-1.99, 4.48) -- (-2.98, 5.24);	
			\draw[thick] (-1.73, 3) -- (-3.14, 2.49);
				\draw[thick] (-3.14, 2.49) -- (-4.29, 2.97);
				\draw[thick] (-3.14, 2.49) -- (-3.72, 1.38);
  \end{scope}
	\end{tikzpicture}
\caption{Markov and Cohn topographs related by the trace map}
\label{fig:CFTree}
\end{center}
\end{figure}

Aigner \cite{Aigner} classified all possible generalisation of the initial matrices $A$ and $B$, preserving the relation with Markov numbers:
\begin{equation}
	\label{AigCohn}
	A (a):= \begin{pmatrix}
		a & 1 \\ 3a-a^2-1 & 3-a
	\end{pmatrix}, \quad
	B(a):= \begin{pmatrix}
		2a+1 & 2 \\ -2a^2+4a+2 & 5-2a
	\end{pmatrix},
\end{equation}
where $a \in \mathbb Z$ is arbitrary. The original Cohn's choice corresponds to $a=1$.

Let 
\begin{equation}
	\label{C_t}
	C_t(a)= \begin{pmatrix}
		a_t & m_t \\ c_t & 3m_t-a_t
	\end{pmatrix} 
\end{equation}
be the Cohn matrix corresponding to $t\in \mathbb Q_{[0,1]}$, such that $C_{\frac{0}{1}}(a)= A(a), \,\, C_{\frac{1}{1}}(a)= B(a).$ Following Aigner \cite{Aigner}, we define the {\it index} of Cohn matrix $C_t(a)$ by
\begin{equation}
	\label{index}
	I_t(a)= \frac{a_t}{m_t}.
\end{equation}
Aigner proved that $I_t(a)$ is monotone on $t\in [0,1]$ and used this to show that for every $a\in \mathbb Z$ all the Cohn matrices $C_t(a)$ in the Cohn tree are distinct. Note that the same claim for the Markov numbers is known as Frobenius Uniqueness Conjecture and still to be proved.

Our first result is the following observation.\footnote{A similar observation for $a=3$ was also made recently by Jouteur et al in \cite{JPRT}.}

\begin{theorem}
\label{main}
The Markov fractions in the fundamental domain $[a, a+\frac{1}{2}]$ coincide with the indices $I_t(a)$ of the corresponding quantum Cohn matrices:
$$
\mu(t)=I_t(a).
$$
\end{theorem}

\begin{proof}
It is enough to prove this for $a=0$ corresponding to the choice of the fundamental domain $[0, \frac{1}{2}]$ of the group $\textrm{Aff}_1(\mathbb Z),$ acting by $x \to \pm x +m, \, m\in \mathbb Z.$

For $a=0$ we have the initial Cohn matrices 
\begin{equation}
	\label{Aig0}
	A (0):= \begin{pmatrix}
		0 & 1 \\ -1 & 3
	\end{pmatrix}, \quad
	B(0):= \begin{pmatrix}
		1 & 2 \\ 2 & 5
	\end{pmatrix},
\end{equation}
and a part of the Conway topograph shown on Fig. 5.

\begin{figure}[h]
\begin{center}
\includegraphics[width=60mm]{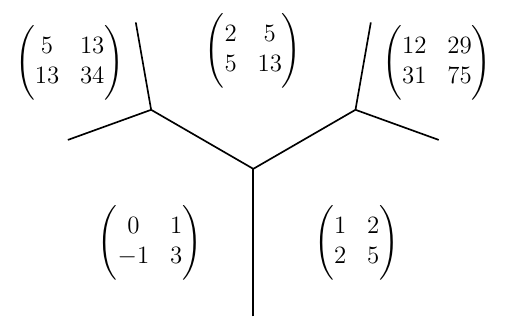} 
\caption{\small Cohn matrices for $a=0.$}
\end{center}
\end{figure}

We prove the theorem by induction. The statement is clearly true for the Cohn matrices shown on Fig. 5.

Assume now that the claim is true for the Markov fractions $\frac{p_1}{q_1}, \frac{p_2}{q_2}, \frac{p_3}{q_3}$ shown on Fig. 3 and prove that then it is true for the fractions $\frac{p_1'}{q_1'},  \frac{p_2'}{q_2'}$ as well. It is enough to prove this for the fraction $\frac{p_1'}{q_1'}.$

By the assumption and properties of the Cohn matrices for $a=0$ we must prove that
$$
\begin{pmatrix}
		p_3 & q_3 \\ c_3 & d_3
	\end{pmatrix}
\begin{pmatrix}
		p_2 & q_2 \\ c_2 & d_2
	\end{pmatrix}
 =\begin{pmatrix}
		p_1' & q_1' \\ c_1'& d_1'
	\end{pmatrix},
	$$
where $c_k, d_k$ are uniquely determined by the conditions $\det C_k=1, \,\, tr \, C_k=3q_k$, holding for all Cohn matrices:
$$
d_k=3q_k-p_k,\,\,\,\, c_k=\frac{3p_kq_k-p_k^2-1}{q_k}.
$$
Since $q_1'$ is the corresponding Markov number, we need to prove only the equality of the first matrix entries of both sides:
$$p_2p_3+q_3c_2=p_1'.$$
By the relation (\ref{rel1'}) for Markov fractions $p_1'=\frac{p_2q_2+p_3q_3}{q_1}$, so we need to show that
$$
p_2p_3+q_3\frac{3p_2q_2-p_2^2-1}{q_2}=\frac{p_2q_2+p_3q_3}{q_1},
$$
or, after multiplication by $q_1q_2$,
$$
p_2p_3q_1q_2+3p_2q_1q_2q_3-p_2^2q_1q_3-q_1q_3-p_2q_2^2-p_3q_2q_3.
$$
Using the Markov relation $q_1^2+q_2^2+q_3^2=3q_1q_2q_3,$ 
we can rewrite this as
$$
p_2p_3q_1q_2+p_2q_1^2+p_2q_3^2-p_2^2q_1q_3-q_1q_3-p_3q_2q_3=0,
$$
or, equivalently, as
$
(p_2q_3-p_3q_2-q_1)(q_3-p_2q_1)=0,
$
which is indeed zero by the property (\ref{rel12}) of the Markov fractions.
\end{proof}

We will use this now to describe the continued fraction expansion of the Markov fractions, which were first described in a more complicated way by Gbur \cite{Gbur}.

\section{Mirror Conway topographs and Markov continued fractions}

Consider now the Conway topograph $\mathcal T_C(2)$ of Cohn matrices with $a=2$ and its mirror image $\mathcal T_C(2)^*$ after reflection in vertical line and taking the transposition of the matrices (see Fig. 6).

\begin{figure}[h]
\begin{center}
    \begin{tikzpicture}[scale=0.65, every node/.style={scale=0.75}]
  \begin{scope}
        \node at (1.75,0.75) {$\displaystyle\left(\begin{array}{cc}
                                                5 & 2 \\
                                                2 & 1 
                                              \end{array}\right)$};
	\node at (-1.75,0.75){$\displaystyle\left(\begin{array}{cc}
                                                2 & 1 \\
                                                1 & 1 
                                              \end{array}\right)$};
	\node at (0,4) {$\displaystyle\left(\begin{array}{cc}
                                                12 & 5 \\
                                               7 & 3 
                                              \end{array}\right)$};
	\node at (3.11, 3.8) {$\displaystyle\left(\begin{array}{cc}
                                               70 & 29 \\
                                               41 & 17
                                              \end{array}\right)$};
	\node at (-3.11, 3.8) {$\displaystyle\left(\begin{array}{cc}
                                               31 & 13 \\
                                               19 & 8
                                              \end{array}\right)$};
                                                       \draw[thick] (0, -0.5) -- (0, 2);
		\draw[thick] (0,2) -- (1.73,3);
			\draw[thick] (1.73, 3) -- (1.99, 4.48);	
				\draw[thick] (1.99, 4.48) -- (1.31, 5.53);
				\draw[thick] (1.99, 4.48) -- (2.98, 5.24);	
			\draw[thick] (1.73, 3) -- (3.14, 2.49);
				\draw[thick] (3.14, 2.49) -- (4.29, 2.97);
				\draw[thick] (3.14, 2.49) -- (3.72, 1.38);
		\draw[thick] (0,2) -- (-1.73,3);
			\draw[thick] (-1.73, 3) -- (-1.99, 4.48);	
				\draw[thick] (-1.99, 4.48) -- (-1.31, 5.53);
				\draw[thick] (-1.99, 4.48) -- (-2.98, 5.24);	
			\draw[thick] (-1.73, 3) -- (-3.14, 2.49);
				\draw[thick] (-3.14, 2.49) -- (-4.29, 2.97);
				\draw[thick] (-3.14, 2.49) -- (-3.72, 1.38);
          \end{scope}
 \begin{scope}[xshift=10cm]
        \node at (-1.75,0.75) {$\displaystyle\left(\begin{array}{cc}
                                                5 & 2 \\
                                                2 & 1 
                                              \end{array}\right)$};
	\node at (1.75,0.75){$\displaystyle\left(\begin{array}{cc}
                                                2 & 1 \\
                                                1 & 1 
                                              \end{array}\right)$};
	\node at (0,4) {$\displaystyle\left(\begin{array}{cc}
                                                12 & 7 \\
                                               5 & 3 
                                              \end{array}\right)$};
	\node at (3.2, 3.8) {$\displaystyle\left(\begin{array}{cc}
                                               31 & 19 \\
                                               13 & 8
                                              \end{array}\right)$};
	\node at (-3.2, 3.8) {$\displaystyle\left(\begin{array}{cc}
                                               70 & 41 \\
                                               29 & 17
                                              \end{array}\right)$};
        \draw[thick] (0, -0.5) -- (0, 2);
		\draw[thick] (0,2) -- (1.73,3);
			\draw[thick] (1.73, 3) -- (1.99, 4.48);	
				\draw[thick] (1.99, 4.48) -- (1.31, 5.53);
				\draw[thick] (1.99, 4.48) -- (2.98, 5.24);	
			\draw[thick] (1.73, 3) -- (3.14, 2.49);
				\draw[thick] (3.14, 2.49) -- (4.29, 2.97);
				\draw[thick] (3.14, 2.49) -- (3.72, 1.38);
		\draw[thick] (0,2) -- (-1.73,3);
			\draw[thick] (-1.73, 3) -- (-1.99, 4.48);	
				\draw[thick] (-1.99, 4.48) -- (-1.31, 5.53);
				\draw[thick] (-1.99, 4.48) -- (-2.98, 5.24);	
			\draw[thick] (-1.73, 3) -- (-3.14, 2.49);
				\draw[thick] (-3.14, 2.49) -- (-4.29, 2.97);
				\draw[thick] (-3.14, 2.49) -- (-3.72, 1.38);
  \end{scope}
	\end{tikzpicture}
\caption{Conway topograph $\mathcal T_C(2)$ and its mirror image}
\label{fig:CFT}
\end{center}
\end{figure}

We can use it to describe the continued fraction expansion of the Markov fractions in the following way.\footnote{Boris Springborn informed me that the continued fraction expansion of the Markov fractions can be extracted also from the work \cite{CS} by Çanakçı and Schiffler in connection with the Markov snake graphs.
A relation between Springborn operation and concatenation of continued fractions was also observed in \cite{JPRT}.}

Note that there are two ways to write the continued fraction expansion of the rational numbers; we will choose the one with even numbers of partial quotients (e.g. $[1,1]$ rather than [2]).

Consider the mirror Conway topograph $\mathcal T_M^*$ of the representatives of Markov fractions in $[2, 5/2]$ shown on the left of Fig. 7.
On the right we have the Conway topograph $\mathcal T_{con}^*$ of the continued fractions constructed from the initial $[22]$, $[11]$ via the concatenation rule
$$
[a_1,\dots, a_{2m}]*[b_1,\dots, b_{2n}]=[a_1,\dots, a_{2m}, \, b_1,\dots, b_{2n}].
$$

\begin{figure}[h]
\begin{center}
    \begin{tikzpicture}[scale=0.65, every node/.style={scale=0.75}]
  \begin{scope}
        \node at (1.75,0.75) {$\displaystyle\frac{2}{1}$};
	\node at (-1.75,0.75){$\displaystyle\frac{5}{2}$};
	\node at (0,4) {$\displaystyle\frac{12}{5}$};
	\node at (3.11, 3.8) {$\displaystyle\frac{31}{13}$};
	\node at (-3.11, 3.8) {$\displaystyle\frac{70}{29}$};
                                                       \draw[thick] (0, -0.5) -- (0, 2);
		\draw[thick] (0,2) -- (1.73,3);
			\draw[thick] (1.73, 3) -- (1.99, 4.48);	
				\draw[thick] (1.99, 4.48) -- (1.31, 5.53);
				\draw[thick] (1.99, 4.48) -- (2.98, 5.24);	
			\draw[thick] (1.73, 3) -- (3.14, 2.49);
				\draw[thick] (3.14, 2.49) -- (4.29, 2.97);
				\draw[thick] (3.14, 2.49) -- (3.72, 1.38);
		\draw[thick] (0,2) -- (-1.73,3);
			\draw[thick] (-1.73, 3) -- (-1.99, 4.48);	
				\draw[thick] (-1.99, 4.48) -- (-1.31, 5.53);
				\draw[thick] (-1.99, 4.48) -- (-2.98, 5.24);	
			\draw[thick] (-1.73, 3) -- (-3.14, 2.49);
				\draw[thick] (-3.14, 2.49) -- (-4.29, 2.97);
				\draw[thick] (-3.14, 2.49) -- (-3.72, 1.38);
          \end{scope}
 \begin{scope}[xshift=10cm]
        \node at (-1.75,0.75) {$\displaystyle[2,2]$};
	\node at (1.75,0.75){$\displaystyle[1,1]$};
	\node at (0,4) {$\displaystyle[2,2,1,1]$};
	\node at (3.4, 3.8) {$\displaystyle[2,2,1,1,1,1]$};
	\node at (-3.4, 3.8) {$\displaystyle[2,2,2,2,1,1]$};
        \draw[thick] (0, -0.5) -- (0, 2);
		\draw[thick] (0,2) -- (1.73,3);
			\draw[thick] (1.73, 3) -- (1.99, 4.48);	
				\draw[thick] (1.99, 4.48) -- (1.31, 5.53);
				\draw[thick] (1.99, 4.48) -- (2.98, 5.24);	
			\draw[thick] (1.73, 3) -- (3.14, 2.49);
				\draw[thick] (3.14, 2.49) -- (4.29, 2.97);
				\draw[thick] (3.14, 2.49) -- (3.72, 1.38);
		\draw[thick] (0,2) -- (-1.73,3);
			\draw[thick] (-1.73, 3) -- (-1.99, 4.48);	
				\draw[thick] (-1.99, 4.48) -- (-1.31, 5.53);
				\draw[thick] (-1.99, 4.48) -- (-2.98, 5.24);	
			\draw[thick] (-1.73, 3) -- (-3.14, 2.49);
				\draw[thick] (-3.14, 2.49) -- (-4.29, 2.97);
				\draw[thick] (-3.14, 2.49) -- (-3.72, 1.38);
  \end{scope}
	\end{tikzpicture}
\caption{Conway topographs $\mathcal T_M^*$ and $\mathcal T_{con}^*$}
\label{fig:MFCF0}
\end{center}
\end{figure}

\begin{theorem}
\label{contfrac}
The continued fraction expansions of the Markov fractions in $[2, 5/2]$ is given by the juxtaposition of the Conway topographs $\mathcal T_M^*$ and $\mathcal T_{con}^*$.
\end{theorem}

\begin{proof}
From Theorem 4.13 of Aigner \cite{Aigner} the Cohn matrix $C_t(a)$ has the form
$$
C_t(a)=\begin{pmatrix}
		a m_t+u_t & m_t \\ c_t & (3-a)m_t-u_t
	\end{pmatrix}
$$
for $m_t=q_t$ being the corresponding Markov number and some $u_t$ and $c_t$.
From our theorem 3.1 it follows that $u_t=p_t$ is the corresponding numerator of the Markov fraction $\mu(t)=\frac{p_t}{q_t}.$
This means that  the ratio of the matrix entries in the first column of 
$$
C_t(2)^\top=\begin{pmatrix}
		2 q_t+p_t & c_t \\ q_t & q_t-p_t
	\end{pmatrix}
$$
equals $\frac{2q_t+p_t}{q_t}=2+\mu(t).$ 
Note that the Markov fractions shifted by 2 are simply the representatives from the different fundamental domain $[2,5/2]$ of the group $\textrm{Aff}_1(\mathbb Z).$

Recall the well-known relation between $2\times 2$-matrices and continued fractions:
$$
M_{[c]}:=\begin{pmatrix}
		c_1 & 1 \\ 1 & 0
	\end{pmatrix}
\begin{pmatrix}
		c_2 & 1 \\ 1 & 0
	\end{pmatrix}
	\dots
	\begin{pmatrix}
		c_{k-1} & 1 \\ 1 & 0
	\end{pmatrix}
\begin{pmatrix}
		c_{k} & 1 \\ 1 & 0
	\end{pmatrix}
	=\begin{pmatrix}
		p_k & p_{k-1} \\ q_k & q_{k-1}
	\end{pmatrix},
$$
where $\frac{p_{k}}{q_{k}}=[c_1,\dots, c_{k}]$ and $\frac{p_{k-1}}{q_{k-1}}=[c_1,\dots, c_{k-1}]$  are the corresponding continued fraction expansions.
The key observation now is that the initial Cohn matrices can be written as
$$
	\begin{pmatrix}
		5 & 2 \\ 2 & 1
	\end{pmatrix}
	=\begin{pmatrix}
		2 & 1 \\ 1 & 0
	\end{pmatrix}
\begin{pmatrix}
		2 & 1 \\ 1 & 0
	\end{pmatrix},
	\quad
	\begin{pmatrix}
		2 & 1 \\ 1 & 1
	\end{pmatrix}
	=\begin{pmatrix}
		1 & 1 \\ 1 & 0
	\end{pmatrix}
\begin{pmatrix}
		1 & 1 \\ 1 & 0
	\end{pmatrix}.
$$

Now the claim follows from the recursive construction of the Cohn matrices and the obvious relation
$
M_{[a]} M_{[b]} =M_{[a]*[b]}.
$
\end{proof}

It is interesting that the  periodic versions of the corresponding continued fractions describe the {\it Markov irrationalities} \begin{equation}
\label{Markovirr}
\gamma\left(\frac{p}{q}\right)=\frac{2p+q+\sqrt{9q^2-4}}{2q},
\end{equation} 
which are special representatives of the quadratic irrationals with Markov constant larger than $\frac{1}{3}$ (see Chapter 2 in \cite{Aigner} and section 4 in \cite{SV1}). 

\begin{figure}[h]
\begin{center}
    \begin{tikzpicture}[scale=0.65, every node/.style={scale=0.75}]
  \begin{scope}
        \node at (1.75,0.75) {$\displaystyle\frac{1+\sqrt{5}}{2}$};
	\node at (-1.75,0.75){$\displaystyle 1+\sqrt 2$};
	\node at (0,4) {$\displaystyle\frac{9+\sqrt{221}}{10}$};
	\node at (3.3, 3.8) {$\displaystyle\frac{23+\sqrt{1517}}{26}$};
	\node at (-3.3, 3.8) {$\displaystyle\frac{53+\sqrt{7565}}{58}$};
                                                       \draw[thick] (0, -0.5) -- (0, 2);
		\draw[thick] (0,2) -- (1.73,3);
			\draw[thick] (1.73, 3) -- (1.99, 4.48);	
				\draw[thick] (1.99, 4.48) -- (1.31, 5.53);
				\draw[thick] (1.99, 4.48) -- (2.98, 5.24);	
			\draw[thick] (1.73, 3) -- (3.14, 2.49);
				\draw[thick] (3.14, 2.49) -- (4.29, 2.97);
				\draw[thick] (3.14, 2.49) -- (3.72, 1.38);
		\draw[thick] (0,2) -- (-1.73,3);
			\draw[thick] (-1.73, 3) -- (-1.99, 4.48);	
				\draw[thick] (-1.99, 4.48) -- (-1.31, 5.53);
				\draw[thick] (-1.99, 4.48) -- (-2.98, 5.24);	
			\draw[thick] (-1.73, 3) -- (-3.14, 2.49);
				\draw[thick] (-3.14, 2.49) -- (-4.29, 2.97);
				\draw[thick] (-3.14, 2.49) -- (-3.72, 1.38);
          \end{scope}
 \begin{scope}[xshift=10cm]
        \node at (-1.75,0.75) {$\displaystyle[\overline{2,2}]$};
	\node at (1.75,0.75){$\displaystyle[\overline{1,1}]$};
	\node at (0,4) {$\displaystyle[\overline{2,2,1,1}]$};
	\node at (3.4, 3.8) {$\displaystyle[\overline{2,2,1,1,1,1}]$};
	\node at (-3.4, 3.8) {$\displaystyle[\overline{2,2,2,2,1,1}]$};
        \draw[thick] (0, -0.5) -- (0, 2);
		\draw[thick] (0,2) -- (1.73,3);
			\draw[thick] (1.73, 3) -- (1.99, 4.48);	
				\draw[thick] (1.99, 4.48) -- (1.31, 5.53);
				\draw[thick] (1.99, 4.48) -- (2.98, 5.24);	
			\draw[thick] (1.73, 3) -- (3.14, 2.49);
				\draw[thick] (3.14, 2.49) -- (4.29, 2.97);
				\draw[thick] (3.14, 2.49) -- (3.72, 1.38);
		\draw[thick] (0,2) -- (-1.73,3);
			\draw[thick] (-1.73, 3) -- (-1.99, 4.48);	
				\draw[thick] (-1.99, 4.48) -- (-1.31, 5.53);
				\draw[thick] (-1.99, 4.48) -- (-2.98, 5.24);	
			\draw[thick] (-1.73, 3) -- (-3.14, 2.49);
				\draw[thick] (-3.14, 2.49) -- (-4.29, 2.97);
				\draw[thick] (-3.14, 2.49) -- (-3.72, 1.38);
  \end{scope}
	\end{tikzpicture}
\caption{Markov irrationalities and their continued fraction expansions}
\label{fig:MFCF1}
\end{center}
\end{figure}

They are the limits of the corresponding {\it left companions} $\gamma_m^-\left(\frac{p}{q}\right)$ of the Markov fraction $\frac{p}{q}\in [2, 5/2]$, defined by Springborn \cite{Springborn}:
$$
\lim_{m\to\infty}\gamma_m^-\left(\frac{p}{q}\right)=\gamma\left(\frac{p}{q}\right).
$$
If $\frac{p}{q}=\gamma_1^\pm\left(\frac{p}{q}\right)=[c_1,\dots, c_{2n}]$ is the continued fraction expansion described above, then the continued fraction expansion of the left companions are
$
\gamma_m^-\left(\frac{p}{q}\right)=[c_1,\dots, c_{2n}]^m,
$
where the power is understood in the sense of concatenation, for example
$$
\gamma_2^-\left(\frac{p}{q}\right)=[c_1,\dots, c_{2n}]^2=[c_1,\dots, c_{2n}, c_1,\dots, c_{2n}].
$$
The continued fraction expansion of the right companions $\gamma_m^+\left(\frac{p}{q}\right)$ and the corresponding Markov irrationalities
looks a bit more involved and needs further analysis, which should be compared with Markov's calculations in \cite{Markov}, geometric approach of Springborn \cite{Springborn} and combinatorial description of Çanakçı and Schiffler \cite{CS}.

\section{$q$-deformed Markov fractions and quantum Cohn matrices}

In 2019 Morier-Genoud and Ovsienko \cite{MGO} introduced a remarkable $q$-deformation of the rational numbers, which surprisingly had not been considered before. 

For a rational $x=\frac{r}{s}>0$ 
the $q$-deformation is a rational function
$
  \left[ \frac{r}s \right]_q = \frac{\mathcal{R}(q)}{\mathcal{S}(q)},
$
where~$\mathcal{R}(q)$ and~$\mathcal{S}(q)$ are certain coprime monic polynomials with positive integer coefficients
both depending on~$r$ and~$s$, which are defined by a certain recursive procedure in terms of the corresponding continued fraction expansion of $x.$

One can visualize these quantum rationals on the Conway topograph \cite{Conway} using the quantized local rules shown in Fig.~3 (taken from \cite{EVW}).
\medskip

\begin{figure}[h]
\begin{center}
	\begin{tikzpicture}
	\node at (-1, 0.75) {$\frac{\mathcal{R}}{\mathcal{S}}$};
	\node at (1, 0.75) {$\frac{\mathcal{R'}}{\mathcal{S'}}$};
	\node at (0, 2.5) {$\frac{\mathcal{R}+q^\ell\mathcal{R'}}{\mathcal{S}+q^\ell\mathcal{S'}}$};
	
	\draw[thick] (0, 0) -- node[right, text=red] {$q^\ell$} (0, 1.5);
	\draw[thick] (0, 1.5) -- (1, 2.5);
	\draw[thick] (0, 1.5) -- (-1, 2.5);
	\end{tikzpicture}
	\quad \quad \quad
	\begin{tikzpicture}
	\node at (-1, 0.75) {$\frac{\mathcal{R}}{\mathcal{S}}$};
	\node at (1, 0.75) {$\frac{\mathcal{R'}}{\mathcal{S'}}$};
	
	\draw[thick] (0, 0) -- node[right, text=red] {$q^\ell$} (0, 1.5);
	\draw[thick] (0, 1.5) -- node[below, right, text=red] {$q^{\ell+1}$} (1, 2.5);
	\draw[thick] (0, 1.5) -- node[above, right, text=red] {$q$} (-1, 2.5);
	\end{tikzpicture}
  \end{center}
\medskip
  \begin{center}
	\begin{tikzpicture} 
    \node at (1.75,0.75) {$\frac{1}{0}$};
	\node at (-1.75,0.75){$\frac{0}{1}$};
	\node at (0,4) {$\frac{1}{1}$};
	\node at (3.11, 3.8) {$\frac{1+q}{1}$};
	\node at (2.30, 5.76) {$\frac{1+q+q^2}{1+q}$}; 
	\node at (4.37, 1.98) {$\frac{1+q+q^2}{1}$}; 
	\node at (0.87, 6.97) {$\frac{1+q+q^2+q^3}{1+q+q^2}$}; 
	\node at (4.38, 5.85) {$\frac{1+q+2q^2+q^3}{1+q+q^2}$}; 
	\node at (5.81, 3.35) {$\frac{1+2q+q^2+q^3}{1+q}$}; 
	\node at (4.43, 0.34) {$\frac{1+q+q^2+q^3}{1}$}; 
	\node at (-3.11, 3.8) {$\frac{q}{1+q}$};
	\node at (-2.10, 5.76) {$\frac{q+q^2}{1+q+q^2}$}; 
	\node at (-4.37, 2.18) {$\frac{q^2}{1+q+q^2}$}; 
	\node at (-0.87, 6.97) {$\frac{q+q^2+q^3}{1+q+q^2+q^3}$}; 
	\node at (-4.38, 5.85) {$\frac{q+q^2+q^3}{1+2q+q^2+q^3}$}; 
	\node at (-5.81, 3.35) {$\frac{q^2+q^3}{1+q+2q^2+q^3}$}; 
	\node at (-4.43, 0.34) {$\frac{q^3}{1+q+q^2+q^3}$}; 
	
	\draw[thick] (0, -0.5) -- node[right, text=red] {\small $1$} (0,2);
	
		\draw[thick] (0,2) -- node[below, text=red] {\small $q$} (1.73,3);
			
			\draw[thick] (1.73, 3) -- node[right, text=red] {\small $q$} (1.99, 4.48);	
	
				\draw[thick] (1.99, 4.48) -- node[above, text=red] {\small $q$} (1.31, 5.53);
				
					\draw[thick] (1.31, 5.53) -- node[above, text=red] {\small $q$} (0.36, 5.87);
					\draw[thick] (1.31, 5.53) -- node[right, text=red] {\small $q^2$} (1.23, 6.53);
					
				\draw[thick] (1.99, 4.48) -- node[below, text=red] {\small $q^2$} (2.98, 5.24);	
				
					\draw[thick] (2.98, 5.24) -- node[right, text=red] {\small $q$} (3.24, 6.21);
					\draw[thick] (2.98, 5.24) -- node[below, text=red] {\small $q^3$} (3.98, 5.24);
	
			\draw[thick] (1.73, 3) -- node[below, text=red] {\small $q^2$} (3.14, 2.49);
	
				\draw[thick] (3.14, 2.49) -- node[below, text=red] {\small $q$} (4.29, 2.97);
	
					\draw[thick] (4.29, 2.97) -- node[right, text=red] {\small $q$} (4.79, 3.84);
					\draw[thick] (4.29, 2.97) -- node[below, text=red] {\small $q^2$} (5.26, 2.71);
				
				\draw[thick] (3.14, 2.49) -- node[left, text=red] {\small $q^3$} (3.72, 1.38);
					
					\draw[thick] (3.72, 1.38) -- node[below, text=red] {\small $q$} (4.63, 0.96);
					\draw[thick] (3.72, 1.38) -- node[left, text=red] {\small $q^4$} (3.55, 0.39);
					
		\draw[thick] (0,2) -- node[above, text=red] {\small $q$} (-1.73,3);
		
			\draw[thick] (-1.73, 3) -- node[right, text=red] {\small $q^2$} (-1.99, 4.48);	
		
				\draw[thick] (-1.99, 4.48) -- node[right, text=red] {\small $q^3$} (-1.31, 5.53);
		
					\draw[thick] (-1.31, 5.53) -- node[below, text=red] {\small $q^4$} (-0.36, 5.87);
					\draw[thick] (-1.31, 5.53) -- node[right, text=red] {\small $q$} (-1.23, 6.53);
		
				\draw[thick] (-1.99, 4.48) -- node[above, text=red] {\small $q$} (-2.98, 5.24);	
		
					\draw[thick] (-2.98, 5.24) -- node[right, text=red] {\small $q^2$} (-3.24, 6.21);
					\draw[thick] (-2.98, 5.24) -- node[above, text=red] {\small $q$} (-3.98, 5.24);
		
			\draw[thick] (-1.73, 3) -- node[above, text=red] {\small $q$} (-3.14, 2.49);
		
				\draw[thick] (-3.14, 2.49) -- node[above, text=red] {\small $q^2$} (-4.29, 2.97);
		
					\draw[thick] (-4.29, 2.97) -- node[right, text=red] {\small $q^3$} (-4.79, 3.84);
					\draw[thick] (-4.29, 2.97) -- node[above, text=red] {\small $q$} (-5.26, 2.71);
		
				\draw[thick] (-3.14, 2.49) -- node[left, text=red] {\small $q$} (-3.72, 1.38);
		
					\draw[thick] (-3.72, 1.38) -- node[above, text=red] {\small $q^2$} (-4.63, 0.96);
					\draw[thick] (-3.72, 1.38) -- node[left, text=red] {\small $q$} (-3.55, 0.39);
	\end{tikzpicture}

\end{center}
\caption{Local rules and the quantized Conway-Farey Topograph} \label{fig:QFCT}
\end{figure}

Evans, Jouteur, Morier-Genoud, Ovsienko \cite{EJMGO} introduced the following $q$-version of Cohn matrices, starting with
\begin{equation}
\label{Cohn_q}
A(a)_q = \begin{pmatrix}  q^{2-a}[a]_q & q^{1-a} \\ [a]_q[3-a]_q-q^{a-1} & q^{-1}[3-a]_q \end{pmatrix}, \qquad
B(a)_q =A(a)_qA(a+1)_q.
\end{equation}
In particular, for $a=0$ we have
$$
A(0)_q = \begin{pmatrix}  0 & q \\ -q^{-1} & q^{-1}+1+q \end{pmatrix}, \qquad
B(0)_q = \begin{pmatrix} q^2 & 1+q \\ q+q^2 & q^{-2}+q^{-1}+2+q \end{pmatrix}.
$$
Let 
\begin{equation}
	\label{C_t^q}
	C_t^q(a)= \begin{pmatrix}
		a_t^q & b_t^q\\ c_t^q & d_t^q
	\end{pmatrix} 
\end{equation}
be the $q$-deformed Cohn matrix corresponding to $t\in \mathbb Q_{[0,1]}]$ with $$C_{\frac{0}{1}}(a)= A(a)_q, \,\, C_{\frac{1}{1}}(a)= B(a)_q$$ 
and define its index by the Aigner formula
\begin{equation}
	\label{indexq}
	I_t^q(a)= \frac{a_t^q}{b_t^q}.
\end{equation}

Modifying the previous analysis and using the results from \cite{MGO}, we have the following quantum version of Theorem 3.1.

\begin{theorem}
$q$-deformed Markov fractions in the fundamental domain $[a, a+\frac{1}{2}]$ are the (scaled by $q$) indices $I_t^q(a)$ of the corresponding quantum Cohn matrices:
$$
\mu(t)_q=q^{-1}I_t^q(a).
$$
\end{theorem}

In particular,
$$
\left[ \frac{1}2 \right]_q=\frac{q}{1+q}, \left[ \frac{2}5 \right]_q=\frac{q^2+q^3}{1+q+2q^2+q^3}, \left[ \frac{5}{13} \right]_q=\frac{q^2+q^3+2q^4+q^5}{1+2q+3q^2+3q^3+3q^4+q^5}.
$$
Note that the denominators are not the quantum Markov numbers defined in \cite{EJMGO, Oguz}, which we will discuss in the next section.

\section{Geometry of quantum Markov numbers}

The usual Markov numbers appear both as off-diagonal entries and as one-third of the trace of the Cohn matrices \cite{Aigner}.
We will use the following definition of $q$-Markov numbers from \cite{EJMGO,Oguz} based on the trace of quantum Cohn matrices: 
\begin{equation}
\label{qMar}
m_t^q:= \frac{q \, tr \, C_t^q(a)}{[3]_q},
\end{equation}
$$
2_q=q^{-1}+q, \, 5_q=q^{-2}+q^{-1}+1+q+q^2, \,13_q=q^{-3}+2q^{-2}+2q^{-1}+3+2q+2q^2+q^3.
$$
These numbers are independent of $a$ and satisfy the {\it $q$-Markov equation} introduced in \cite{Kogiso}
\begin{equation}
\label{qMarK}
x^2+y^2+z^2=q^{-1}[3]_qxyz+(q-1)(q^{-1}-1).
\end{equation}

To explain their geometric meaning recall that the original Cohn matrices generate the commutator subgroup $G=\Gamma'$ of $\Gamma=SL(2,\mathbb Z)$, which is a free group with two generators $A$ and $B$ given by (\ref{Cohn_1}) (see \cite{Aigner}).

Consider now its deformation $G_q(a)$ generated by the quantized Cohn matrices (\ref{Cohn_q}). From the results of \cite{MGOV,Scherich} it follows that for real positive $q\in \mathbb R$ this is a discrete subgroup $G_q(a) \subset SL(2,\mathbb R).$ 

\begin{theorem}
The corresponding quotient of the hyperbolic plane
$
T_q=\mathbb H/G_q(a)
$
is the hyperbolic torus with 3-fold symmetry and the hole of length $l_q$, where
\begin{equation}
	\label{lengthq}
	\cosh\, \frac{l_q}{2}=1+\frac{1}{2q^3}[3]_q^2(q-1)^2.
\end{equation}
The $q$-Markov numbers are 
\begin{equation}
	\label{lengths}
	m_t^q=\frac{2q}{[3]_q}\cosh\, \frac{l_t}{2},
	\end{equation}
	 where $l_t$ are the lengths of the corresponding closed geodesics on $T_q$.
\end{theorem}

\begin{proof}
Note that in the classical case $q=1$ it is well known that the corresponding quotient is the so-called Markov hyperbolic punctured torus and the claim is known since the work of Cohn and Gorshkov \cite{Cohn, Gorshkov} (see more details in \cite{Aigner}).

The proof is based on the classical {\it Fricke identity}: for any $A,B \in SL_2(\mathbb R), C=AB$ 
\begin{equation}
\label{fricke}
(\mathrm{tr}\, A)^2+(\mathrm{tr}\, B)^2+(\mathrm{tr}\, C)^2=\mathrm{tr}\, A \, \mathrm{tr}\, B\, \mathrm{tr}\,C +\mathrm{tr}\, (ABA^{-1}B^{-1})+2.
\end{equation}
The puncture condition means that the commutator of the generators is a parabolic element with 
$$
\mathrm{tr}\, (ABA^{-1}B^{-1})=-2,
$$
so  (\ref{fricke}) implies that the corresponding $X=\mathrm{tr}\, A,\, Y=\mathrm{tr}\, B,\, Z=\mathrm{tr}\, C$ satisfy the  real version of the Markov equation
\begin{equation}
\label{xyz}	
X^2 + Y^2 + Z^2=XYZ,\quad X,Y,Z \in \mathbb R.
\end{equation}
One can show that the positive component of this surface with $X,Y,Z>0$ is the Teichm\"uller space of the punctured tori (see \cite{Keen}).
The Markov numbers are related to the lengths $l$ of simple closed geodesics by
$m=\frac{2}{3} \cosh \frac{l}{2}$ (see \cite{Aigner,Goldman,Haas}).
  
In the quantum case one can check that 
$$
\mathrm{tr}\, (ABA^{-1}B^{-1})=-q^{-3}[4]_q^2+2q^{-2}[3]_q^2-q^{-1}[2]_q^2,
$$
so
$$
\mathrm{tr}\, (ABA^{-1}B^{-1})+2=q^{-2}[3]_q(q-1)(q^{-1}-1),
$$
leading to the $q$-Markov equation (\ref{qMarK}) for the corresponding $q$-Markov numbers (\ref{qMar}).

The length $l$ of the closed geodesic corresponding to matrix $C \in G_q$ is given by the standard formula
$
2\cosh \frac{l}{2}=\mathrm{tr} C,
$
which implies (\ref{lengths}).
Similarly, the length of the hole $l_q$ satisfy the relation
$$
2\cosh \frac{l_q}{2}=-\mathrm{tr}\, (ABA^{-1}B^{-1})=2+q^{-3}[3]_q^2(q-1)^2,
$$
in agreement with (\ref{lengthq}). 
\end{proof}

\section{Metallic generalisations}

Following \cite{SV1}, consider now the following ``metallic" generalisation $\gamma_t(n), \,\, n\in \mathbb N$ of Markov irrationalities in Farey parametrisation $t \in [0, \frac{1}{2}]$ shown on Fig. 10.  

\begin{figure}[h]
\begin{center}
    \begin{tikzpicture}[scale=0.65, every node/.style={scale=0.75}]
  \begin{scope}
      \node at (1.75,0.75) {$\frac{0}{1}$};
	\node at (-1.75,0.75){$\frac{1}{2}$};
	\node at (0,4) {$\frac{1}{3}$};
	\node at (4.11, 3.8) {$\frac{1}{4}$};
	\node at (-4.11, 3.8) {$\frac{2}{5}$};
                                                       \draw[thick] (0, -0.5) -- (0, 2);
		\draw[thick] (0,2) -- (1.73,3);
			\draw[thick] (1.73, 3) -- (1.99, 4.48);	
				\draw[thick] (1.99, 4.48) -- (1.31, 5.53);
				\draw[thick] (1.99, 4.48) -- (2.98, 5.24);	
			\draw[thick] (1.73, 3) -- (3.14, 2.49);
				\draw[thick] (3.14, 2.49) -- (4.29, 2.97);
				\draw[thick] (3.14, 2.49) -- (3.72, 1.38);
		\draw[thick] (0,2) -- (-1.73,3);
			\draw[thick] (-1.73, 3) -- (-1.99, 4.48);	
				\draw[thick] (-1.99, 4.48) -- (-1.31, 5.53);
				\draw[thick] (-1.99, 4.48) -- (-2.98, 5.24);	
			\draw[thick] (-1.73, 3) -- (-3.14, 2.49);
				\draw[thick] (-3.14, 2.49) -- (-4.29, 2.97);
				\draw[thick] (-3.14, 2.49) -- (-3.72, 1.38);
          \end{scope}
 \begin{scope}[xshift=12cm]
       \node at (1.75,0.75) {$[\overline{n,n}]$};
	\node at (-1.75,0.75){$[\overline{2n,2n}]$};
	\node at (0,4) {$[\overline{2n,2n, n,n}]$};
	\node at (4.11, 3.8) {$[\overline{2n,2n,n,n, n,n}]$};
	\node at (-4.11, 3.8) {$[\overline{2n,2n,2n,2n,n,n}]$};
                                                       \draw[thick] (0, -0.5) -- (0, 2);
		\draw[thick] (0,2) -- (1.73,3);
			\draw[thick] (1.73, 3) -- (1.99, 4.48);	
				\draw[thick] (1.99, 4.48) -- (1.31, 5.53);
				\draw[thick] (1.99, 4.48) -- (2.98, 5.24);	
			\draw[thick] (1.73, 3) -- (3.14, 2.49);
				\draw[thick] (3.14, 2.49) -- (4.29, 2.97);
				\draw[thick] (3.14, 2.49) -- (3.72, 1.38);
		\draw[thick] (0,2) -- (-1.73,3);
			\draw[thick] (-1.73, 3) -- (-1.99, 4.48);	
				\draw[thick] (-1.99, 4.48) -- (-1.31, 5.53);
				\draw[thick] (-1.99, 4.48) -- (-2.98, 5.24);	
			\draw[thick] (-1.73, 3) -- (-3.14, 2.49);
				\draw[thick] (-3.14, 2.49) -- (-4.29, 2.97);
				\draw[thick] (-3.14, 2.49) -- (-3.72, 1.38);

  \end{scope}
	\end{tikzpicture}
\caption{Metallic Markov irrationalities in Farey parametrisation}
\label{fig:MFCF1}
\end{center}
\end{figure}

%
The numbers $$\varphi_n=[\overline{n,n}]=\frac{n+\sqrt{n^2+4}}{2}$$ are known as {\it metallic numbers} with
$\varphi_1=\frac{1+\sqrt{5}}{2}$ being golden ratio, which explains the terminology.

The motivation for introducing this generalisation came from the study of the Lyapunov spectrum of Markov and Euclid trees \cite{SV1}.

Let $x=[a_1,a_2,\dots]$ be the continued fraction expansion, $C_k(x)=\frac{p_k}{q_k}:=[a_1,\dots,a_k]$ be the corresponding convergent, then the growth of $q_k$ on the Conway topograph is described by the {\it Lyapunov exponent}
$$
\Lambda(x):=\limsup_{k\to\infty}\frac{\ln q_k(x)}{s_k(x)}, \quad s_k(x):=a_1+\dots+a_k.
$$

\begin{theorem}\cite{SV1}
Lyapunov function $\Lambda(x)$ has the following properties:
\begin{itemize}
\item $\Lambda(x)$ is defined for all $x \in \mathbb{R}P^1$ and is $GL_2(\mathbb Z)$-invariant:
$$
\Lambda\left(\frac{ax+b}{cx+d}\right)=\Lambda(x), \quad x \in \mathbb{R}P^1,\, \begin{pmatrix}
  a & b \\
  c & d \\
\end{pmatrix} \in GL_2(\mathbb Z)
$$
\item
$\Lambda(x)=0$ almost everywhere (but Hausdorff dimension of its support is 1)
\item
Lyapunov spectrum
$Spec_\Lambda:=\{\Lambda(x), \,\, x \in  \mathbb{R}P^1\}=[0, \ln \varphi], \,\,\varphi=\frac{1+\sqrt 5}{2}$.
\item
The restriction $\psi_n(t):=\Lambda(\gamma_t(n))$ of Lyapunov function on the set of metallic Markov irrationalities 
 can be extended to continuous convex function of $t\in [0,\, \frac{1}{2}]$.
\end{itemize}
\end{theorem}

For $n=1$ we have $\psi_1(t)=\frac{1}{2}\psi(t)$, where $\psi(t)$ is {\it Fock function} \cite{Fock} defined as 
$$
\psi (t) = \frac{1}{q} \cosh^{-1} \left(\frac{3}{2} m_t\right), \,\, t=p/q,
$$
where $m_t$ are the usual Markov numbers.
For general $n$ one can use the theory of Federer-Gromov stable norm to prove that $\psi_n(t)$ is differentiable at all irrational and not differentiable at rational $t \in [0, \frac{1}{2}]$ (see \cite{SorV}).

The corresponding metallic generalisation $\mathcal T^{(n)}_C(2)^*$ of Cohn topograph (which is the conjugated version of the Cohn tree proposed in \cite{SV1}) is shown on Fig. 11 and starts with the matrices 
\begin{equation}
\label{cohna}
A(n) = 
\begin{pmatrix}
 4n^2+1 & 2n \\
  2n & 1 \\
\end{pmatrix},
 \quad B(n) = \begin{pmatrix}
 n^2+1 & n \\
  n & 1 \\
\end{pmatrix}.
\end{equation}
%


\begin{figure}[h]
\begin{center}
	\begin{tikzpicture} [scale=0.7, every node/.style={scale=0.7}]
	\node at (1.75,0.75) 
	{$\begin{pmatrix}
		n^2+1 & n \\ n & 1
	\end{pmatrix}$};
	
	\node at (1.75,-0.05) 
	{{\color{blue} $n^2+2$}};

	\node at (-1.75,0.75) {$\begin{pmatrix}
		4n^2+1 & 2n \\ 2n & 1
	\end{pmatrix}$};
	
	\node at (-1.75,-0.05) 
	{{\color{blue} $4n^2+2$}};
	
	\node at (0,5.25) {$\begin{pmatrix}
		4n^4+7n^2+1 & 4n^3+3n \\ 2n^3+3n & 2n^2+1
	\end{pmatrix}$};

	\node at (0,4.45) 
	{{\color{blue} $4n^4+9n^2+2$}};
	
	\node at (6.11, 3.8) {$\begin{pmatrix}
		4n^6+15n^4+11n^2+1 & n(4n^4+11n^2+4) \\ n(2n^4+7n^2+4) & 2n^4+5n^2+1
	\end{pmatrix}$};	
	
	\node at (6.11,3.0) 
	{{\color{blue} $4n^6+17n^4+16n^2+2$}};
	
	\node at (-6.5, 3.8)	{$\begin{pmatrix}
		16n^6+36n^4+17n^2+1 & n(16n^4+20n^2+5) \\ n(8n^4+16n^2+5) & 8n^4+8n^2+1
	\end{pmatrix}$};
	\draw[thick] (0, -0.5) -- (0,2);
	
	\node at (-6.5,3.0) 
	{{\color{blue} $16n^6+44n^4+25n^2+2$}};

		\draw[thick] (0,2) -- (1.73,3);
			
			\draw[thick] (1.73, 3) -- (2.5, 5.25);	
	
				
					
				
	
			\draw[thick] (1.73, 3) -- (3.96, 2.19);
	
	
				
					
					
		\draw[thick] (0,2) -- (-1.73,3);
		
			\draw[thick] (-1.73, 3) -- (-2.5, 5.25);	
		
		
		
		
		
			\draw[thick] (-1.73, 3) -- (-3.96, 2.19);
		
		
		
		
	\end{tikzpicture}
\end{center}
\label{fig:Cohn_n}
\caption{Metallic Cohn matrices and corresponding Markov numbers}
\end{figure}

Let 
\begin{equation}
	\label{Cohnm)}
	C_t(n)= \begin{pmatrix}
		a_t(n) & b_t(n)\\ c_t(n) & d_t(n)
	\end{pmatrix} 
\end{equation}
be the Cohn matrix corresponding to $t\in \mathbb Q_{[0,\frac{1}{2}]}$ with $C_{\frac{1}{2}}(n)= A(n), \,\, C_{\frac{0}{1}}(n)= B(n).$ 

Motivated by the previous results, we define the {\it metallic Markov numbers} as the traces
\begin{equation}
	\label{metM}
	M_t(n):= tr\, C_t(n)=a_t(n)+d_t(n),
\end{equation}
and the {\it metallic Markov fractions} as
\begin{equation}
	\label{indexn}
	\mu_t(n):= \frac{a_t(n)}{c_t(n)}.
\end{equation}
 When $n=1$ we have the Markov fractions from $[2, 5/2]$ and the usual Markov numbers multiplied by 3. Note that for $n>1$ the denominators of the metallic Markov fractions are not related to the corresponding Markov numbers in a simple way (see Fig.11).

\begin{theorem}\cite{SV1}
The corresponding metallic Markov triples on the Conway topograph satisfy the modified Markov equation
\begin{equation}
	\label{metMar}
	X^2+Y^2+Z^2=XYZ+4-4n^6
	\end{equation}
with the initial triple $X=Y=n^2+2, \, Z=4n^2+2.$
\end{theorem}

This follows from the Fricke identity (\ref{fricke}) since $tr\, ABA^{-1}B^{-1}=2-4n^6$  for the initial matrices $A=A(n), \, B=B(n)$ (and hence for all other neighbours on the metallic Cohn topograph).
The metallic Markov numbers are part of the metallic Markov triples, which can be computed recursively from the initial triple using the Vieta involution
$
(X,Y,Z) \rightarrow (X,Y, XY-Z)
$
and permutations. Note that for $n>1$ there could be other integer solutions of (\ref{metMar}).

In terms of the metallic Markov numbers the values of the Lyapunov exponents of the metallic Markov irrationalities can be given as the following version of the Fock function
\begin{equation}
	\label{psin}
\Lambda(\gamma_t(n))=\psi_n(t)= \frac{1}{2nq} \cosh^{-1} \left(\frac{1}{2} M_t(n)\right),  \quad t=\frac{p}{q}.
\end{equation}
Geometrically the equation (\ref{metMar}) describes the lengths of the simple closed geodesics on the metallic version of hyperbolic Markov torus $T_n=\mathbb H/<A(n),B(n)>$ with a hole of length 
$l=2\cosh^{-1} (2n^6-1)$ (see \cite{SorV,SV1}).

The Conway topograph $\mathcal T_M^*$ of the metallic Markov fractions is shown on the top half of Fig.~\ref{fig:MFCF}.
On the bottom half we have the Conway topograph $\mathcal T_{con}^*$ of the continued fractions constructed from the initial $[2n, 2n]$, $[n, n]$ via the concatenation rule
$$
[a_1,\dots, a_{2m}]*[b_1,\dots, b_{2n}]=[a_1,\dots, a_{2m}, \, b_1,\dots, b_{2n}].
$$

\begin{figure}[h]
\begin{center}
    \begin{tikzpicture}[scale=0.57, every node/.style={scale=0.75}]
  \begin{scope}
        \node at (1.75,0.75) {$\displaystyle\frac{n^2+1}{n}$};
	\node at (-1.75,0.75){$\displaystyle\frac{4n^2+1}{2n}$};
	\node at (0,4) {$\displaystyle\frac{4n^4+7n^2+1}{n(2n^2+3)}$};
	\node at (4.61, 3.8) {$\displaystyle\frac{4n^6+15n^4+11n^2+1}{n(2n^4+7n^2+4)}$};
	\node at (-4.61, 3.8) {$\displaystyle\frac{16n^6+36n^4+17n^2+1}{n(8n^4+16n^2+5)}$};
                                                       \draw[thick] (0, -0.5) -- (0, 2);
		\draw[thick] (0,2) -- (1.73,3);
			\draw[thick] (1.73, 3) -- (1.99, 4.48);	
				\draw[thick] (1.99, 4.48) -- (1.31, 5.53);
				\draw[thick] (1.99, 4.48) -- (2.98, 5.24);	
			\draw[thick] (1.73, 3) -- (3.14, 2.49);
				\draw[thick] (3.14, 2.49) -- (4.29, 2.97);
				\draw[thick] (3.14, 2.49) -- (3.72, 1.38);
		\draw[thick] (0,2) -- (-1.73,3);
			\draw[thick] (-1.73, 3) -- (-1.99, 4.48);	
				\draw[thick] (-1.99, 4.48) -- (-1.31, 5.53);
				\draw[thick] (-1.99, 4.48) -- (-2.98, 5.24);	
			\draw[thick] (-1.73, 3) -- (-3.14, 2.49);
				\draw[thick] (-3.14, 2.49) -- (-4.29, 2.97);
				\draw[thick] (-3.14, 2.49) -- (-3.72, 1.38);
          \end{scope}
 \begin{scope}[yshift=-7cm]
        \node at (-1.75,0.75) {$\displaystyle[2n,2n]$};
	\node at (1.75,0.75){$\displaystyle[n,n]$};
	\node at (0,4) {$\displaystyle[2n,2n,n,n]$};
	\node at (3.9, 3.8) {$\displaystyle[2n,2n,n,n,n,n]$};
	\node at (-4.1, 3.8) {$\displaystyle[2n,2n,2n,2n,n,n]$};
        \draw[thick] (0, -0.5) -- (0, 2);
		\draw[thick] (0,2) -- (1.73,3);
			\draw[thick] (1.73, 3) -- (1.99, 4.48);	
				\draw[thick] (1.99, 4.48) -- (1.31, 5.53);
				\draw[thick] (1.99, 4.48) -- (2.98, 5.24);	
			\draw[thick] (1.73, 3) -- (3.14, 2.49);
				\draw[thick] (3.14, 2.49) -- (4.29, 2.97);
				\draw[thick] (3.14, 2.49) -- (3.72, 1.38);
		\draw[thick] (0,2) -- (-1.73,3);
			\draw[thick] (-1.73, 3) -- (-1.99, 4.48);	
				\draw[thick] (-1.99, 4.48) -- (-1.31, 5.53);
				\draw[thick] (-1.99, 4.48) -- (-2.98, 5.24);	
			\draw[thick] (-1.73, 3) -- (-3.14, 2.49);
				\draw[thick] (-3.14, 2.49) -- (-4.29, 2.97);
				\draw[thick] (-3.14, 2.49) -- (-3.72, 1.38);
  \end{scope}
	\end{tikzpicture}
\caption{Conway topograph of the metallic Markov fractions}
\label{fig:MFCF}
\end{center}
\end{figure}

\begin{theorem}
\label{contfracmet}
The continued fraction expansions of the metallic Markov fractions is given by the juxtaposition of the metallic Conway topographs  $\mathcal T_M^*(n)$ and $\mathcal T_{con}^*(n).$
\end{theorem}

The proof is a straightforward generalisation from the case $n=1$ considered above.

Finally, the $q$-version of the metallic Markov fractions can be described in a similar way using the following {\it quantum metallic Cohn matrice}s 
\begin{gather}
\label{qcohna}
A_q^{(n)}=R_q^{2n} L_q^{2n} = 
\begin{pmatrix}
 q^{2n} + [2n]_q [2n]_{q^{-1}} & q^{-2n} [2n]_q \\
  [2n]_{q^{-1}} & q^{-2n} \\
\end{pmatrix}, \\
 \quad B_q^{(n)}= R_q^n L_q^n = \begin{pmatrix}
 q^{n} + [n]_q [n]_{q^{-1}} & q^{-n} [n]_q \\
 [n]_{q^{-1}} & q^{-n} \\
\end{pmatrix}
\end{gather}
(see more detail in \cite{EV}).
The $q$-version of metallic numbers and their properties were studied by Ren \cite{Ren}  and, more recently, by Pedon \cite{Pedon}.

\section*{Acknowledgements}

I am very grateful to Sam Evans for many useful discussions and help during the preparation of this paper. I would also like to thank Alexander Thomas and Emine Yildirim for pointing out to the very relevant papers \cite{JPRT} and \cite{Oguz} and Boris Springborn for the helpful comments.



\begin{thebibliography}{99}
	
\bibitem{Aigner}
M. Aigner
{\it Markov's Theorem and 100 Years of the Uniqueness Conjecture: A Mathematical Journey from Irrational Numbers to Perfect Matchings}. Springer, 2013.

\bibitem{CS}
İ.  Çanakçı, R. Schiffler {\it Snake graphs and continued fractions.}
Eur. J. Comb. {\bf 86} (2020), article 103081.

 \bibitem{Cohn}
 H. Cohn {\it Approach to Markoff's minimal forms through modular functions}.
  Ann. Math. {\bf 61}(1955),1-12.
	
 \bibitem{Conway}
 J.H. Conway {\it The Sensual (Quadratic) Form}, Carus Mathematical Monographs, Vol.26.
MAA, 1997.

\bibitem{EJMGO}
S. Evans, P. Jouteur, S. Morier-Genoud, V. Ovsienko
{\it On $q$-deformed Markov numbers. Cohn matrices and
perfect matchings with weighted edges.} arXiv:2507.19080.

\bibitem{EVW}
S.J. Evans, A.P. Veselov, B. Winn {\it Quantum Kronecker fractions.} JXM {\bf 2(1)} (2026), 63–80.

\bibitem{EV}
S.J. Evans, A.P. Veselov {\it Quantum metallic Markov fractions.} In preparation.
 
\bibitem{Fock}
  V.V. Fock {\it Dual Teichm\"{u}ller Spaces}. arxiv:dg-ga/9702018v3, 1997.

\bibitem{Gbur}
M.E. Gbur {\it On the minimum of zero indefinite binary quadratic forms.} Mathematika
{\bf 25(1)}(1978), 94-106.

\bibitem{Goldman}
W.M. Goldman {\it The modular group action on real characters of a one-holed torus.} Geometry and Topology {\bf 7} (2003), 443-486.
 
\bibitem{Gorshkov}
D.S. Gorshkov {\it Geometry of Lobachevskii in connection with certain questions of arithmetic.} PhD Thesis, 1953 (in Russian). Zap. Nauch. Sem. LOMI {\bf 67} (1977), 39-85. English transl. in J. Soviet Math. {\bf 16} (1981), 788-820.
 
 \bibitem{Haas}
A. Haas {\it Diophantine approximation on hyperbolic Riemann surfaces.} Acta Math. {\bf156} (1986), 33-82.
 
\bibitem{HP} 
P. Hacking, Y. Prokhorov {\it Smoothable del Pezzo surfaces with quotient singularities.} Compositio Math. {\bf 146} (2010), 169-192.



\bibitem{JPRT}
P. Jouteur, O. Paris-Romaskevich, A. Thomas 
{\it Plane geometry of $q$-rationals and Springborn operations.} arxiv.2603.04295.

\bibitem{Keen} 
L. Keen {\it  A rough fundamental domain for Teichm\"uller spaces.} Bulletin of the AMS {\bf 83(6)} (1977), 1199-1226.

\bibitem{Kogiso}
T. Kogiso {\it $q$-deformations and $t$-deformations of the Markov triples.} arXiv:2008.12913.


\bibitem{Markov}
 A.A. Markov {\it Sur les formes quadratiques binaires ind\'efinies}.
	Mathematische Annalen, {\bf 15} (1879), 381-406; {\bf 17} (1880), 379-399.
	
\bibitem{MGO}
S. Morier-Genoud and V. Ovsienko
{\it $q$-deformed rationals and {$q$}-continued fractions.}
Forum Math.\ Sigma {\bf 8} (2020), e13, 55pp.

\bibitem{MGOV}
S. Morier-Genoud, V. Ovsienko and A.P. Veselov
{\it Burau representation of braid groups
and $q$-rationals.} IMRN {\bf 10} (2024),  8618-8627.


\bibitem{Oguz}
E. Kantarcı Oguz {\it Oriented posets, rank matrices and q-deformed Markov numbers.} Discrete Mathematics,
348(2) (2025), 114256.

\bibitem{LMG}
L. Leclere, S. Morier-Genoud {\it $q$-deformations in the modular group and of the real quadratic irrational
numbers.} Adv. Appl. Math. {\bf 130} (2021) Paper No. 102223, 28 pp.
	
\bibitem{Pedon}
E. Pedon {\it Analytic properties of $q$-metallic numbers.} arXiv:2604.19898.

\bibitem{Ren}
X. Ren {\it On radiuses of convergence of $q$-metallic numbers and related $q$-rational numbers.} Res.
Number Theory {\bf 8:3} (2022), Paper No. 37, 14 pages.

\bibitem{Rud}
A.N. Rudakov {\it The Markov numbers and exceptional bundles
on $\mathbb P^2.$} Mathematics of the USSR-Izvestiya {\bf 32:1} (1989), 99–112.

\bibitem{Scherich}
N. Scherich {\it Classification of the real discrete specializations of the Burau representation of $B_3$.} 
Math. Proc. Camb. Phil. Soc. {\bf 168} (2020).



\bibitem{SorV}
A. Sorrentino, A.P. Veselov {\it Markov numbers, Mather's beta-function and stable norm.} Nonlinearity, {\bf 32:6} (2019), 2147--2156.


\bibitem{SV1}
K. Spalding and A.P. Veselov {\it Lyapunov spectrum of Markov and Euclid trees.} Nonlinearity {\bf 30} (2017), 4428-53.

\bibitem{Springborn}
B. Springborn {\it The worst approximable rational numbers.} J. Number Theory {\bf 263} (2024), 153-205.

\bibitem{V25}
A.P. Veselov  {\it Markov fractions and the slopes of the exceptional bundles on $\mathbb P^2$.} arXiv:2501.06779. 

\end{thebibliography}
\end{document}